\documentclass[12pt]{amsart}

\usepackage{amssymb, amscd, txfonts}
\usepackage{graphicx}
\usepackage[all]{xy} 
\usepackage{mathrsfs}

\numberwithin{equation}{section}

\newtheorem{theorem}{Theorem}[section]
\newtheorem{proposition}[theorem]{Proposition}
\newtheorem{lemma}[theorem]{Lemma}
\newtheorem{corollary}[theorem]{Corollary}

\theoremstyle{definition}
\newtheorem{definition}[theorem]{Definition}
\newtheorem{example}[theorem]{Example}

\theoremstyle{remark}
\newtheorem{remark}[theorem]{Remark}
\newtheorem{claim}[theorem]{Claim}

\newcommand{\N}{\mathbb{N}}
\newcommand{\Z}{\mathbb{Z}}

\newcommand{\Q}{\mathbb{Q}}
\newcommand{\R}{\mathbb{R}}
\newcommand{\C}{\mathbb{C}}

\newcommand{\proj}{{\mathbb P}}

\newcommand{\OL}{{\rm O}^{+}(L)}
\newcommand{\G}{\Gamma}
\newcommand{\D}{\mathcal{D}}
\newcommand{\FG}{\mathcal{F}(\Gamma)}
\newcommand{\FGcpt}{\mathcal{F}(\Gamma)^{\Sigma}}
\newcommand{\FGo}{\mathcal{F}(\Gamma)^{\circ}}
\newcommand{\GN}{\Gamma_{L}[N]}
\newcommand{\FN}{\mathcal{F}_{L}(N)}
\newcommand{\UIZ}{U(I)_{\mathbb{Z}}}
\newcommand{\UJZ}{U(J)_{\mathbb{Z}}}
\newcommand{\rkY}{{\rm rk}(Y)}

\begin{document}

\title[]{Special holomorphic tensors on orthogonal modular varieties and applications to the Lang conjecture}
\author[]{Shouhei Ma}
\thanks{Supported by KAKENHI 21H00971 and 20H00112.} 
\address{Department~of~Mathematics, Tokyo~Institute~of~Technology, Tokyo 152-8551, Japan}
\email{ma@math.titech.ac.jp}

\begin{abstract}
We introduce a method to construct special holomorphic tensors on orthogonal modular varieties 
from scalar-valued modular forms, 
and give applications to the Lang conjecture on the birational type of subvarieties of orthogonal modular varieties. 
\end{abstract} 

\maketitle


\section{Introduction}\label{sec: intro}

Holomorphic tensors on a smooth projective variety $X$, 
namely sections of $(\Omega_{X}^{1})^{\otimes m}$, 
are fundamental birational invariants of $X$. 
While pluricanonical forms would be usually the first invariants to look at, 
more general holomorphic tensors are also used when exploring the fine geometry of $X$. 
In this article we construct holomorphic tensors on modular varieties of orthogonal type 
by using \textit{scalar-valued} modular forms, 
and give applications to the birational type of subvarieties. 

Let $L$ be an integral lattice of signature $(2, n)$ with $n\geq 3$, 
and ${\OL}$ be the index $\leq 2$ subgroup of the orthogonal group ${\rm O}(L)$ 
that preserves the Hermitian symmetric domain ${\D}$ attached to $L$.  
If ${\G}$ is a finite-index subgroup of ${\OL}$, 
it defines a quasi-projective variety ${\FG}={\G}\backslash {\D}$ of dimension $n$. 
Modular varieties of this type appear as moduli spaces of $K3$ surfaces and holomorphic symplectic varieties. 
It has been realized that they tend to be of general type, especially in higher dimension (\cite{GHS1}, \cite{GHS2}, \cite{Ma1}). 

Our central construction in this article attaches 
to a scalar-valued ${\G}$-modular form $F$ of weight $pm$ with $0< p \leq n$ and $m$ even 
a holomorphic tensor $\omega(F)$ of type $(\Omega^{p})^{\otimes m}$ on a Zariski open set of ${\FG}$. 
This gives a generalization of the well-known correspondence between 
$m$-canonical forms and modular forms of weight $nm$. 
We establish an extension criterion for $\omega(F)$ over a smooth projective model of ${\FG}$ 
in terms of the modular form $F$ (\S \ref{sec: extension}). 
In spite of the situation that we are dealing with holomorphic tensors, 
this exhibits a perfect analogy with the case of pluricanonical forms (\cite{GHS1}). 
At the same time, it presents a new problem as regards to the singularities.

Our application in this article is to the birational type of subvarieties of ${\FG}$. 
According to a conjecture of Lang \cite{La}, when ${\FG}$ is of general type, 
there would exist an algebraic subset of ${\FG}$ that contains all subvarieties of ${\FG}$ of non-general type. 
We give some effective answers to this problem. 

Let $Y$ be a subvariety of ${\FG}$ not contained in the branch locus of ${\D}\to {\FG}$. 
We define the \textit{rank} ${\rkY}$ of $Y$ to be the rank of the restriction of a canonical quadratic form on $T_{x}{\D}$ 
to the subspace $T_{x}Y\subset T_{x}{\D}$ for a general point $x$ of $Y$ (\S \ref{ssec: rank}). 
We need to consider this invariant when restricting our holomorphic tensors $\omega(F)$ to $Y$. 
We have 
\begin{equation*}
\dim(Y) - {\rm codim} (Y) \leq {\rkY} \leq \dim (Y). 
\end{equation*}
When ${\rkY}=\dim (Y)$, we say that $Y$ is \textit{nondegenerate}. 
Sub orthogonal modular varieties are examples of nondegenerate subvarieties, 
while ball quotients embedded in ${\FG}$ have rank $0$. 
We expect that most subvarieties $Y$ with $\dim (Y) > n/2$ would be nondegenerate. 

Let $Z_{p}\subset {\FG}$ (resp.~$Z_{p}'\subset {\FG}$) be the common zero locus of 
${\G}$-modular forms whose 
\textit{reflective slope}, \textit{cusp slope} and \textit{elliptic slope} satisfy $<p$ (resp.~$\leq p$). 
See \S \ref{ssec: slope} for the definition of various slopes, 
which are more or less the ratio of weight and vanishing order at the ramification divisor, 
the boundary divisor of a toroidal compactification, and the exceptional divisors over certain bad singularities. 
We have 
\begin{equation*}
Z_1 \supset Z_1' \supset Z_2 \supset Z_2' \supset \cdots \supset Z_n \supset Z_n'. 
\end{equation*} 
Then we prove the following criterion. 

\begin{theorem}[Theorem \ref{thm: full}]\label{thm: intro full}
Let $L$ be a lattice of signature $(2, n)$ with $n\geq 3$ and 
${\G}$ be a finite-index subgroup of ${\rm O}^{+}(L)$. 
Let $Y$ be a subvariety of ${\FG}$ with ${\rkY}\geq p$ not contained in the branch locus. 

(1) If $Y$ is nondegenerate and $Y\not\subset Z_{p}$, $Y$ is of general type. 

(2) If $Y\not\subset Z_{p}$, $Y$ is of Freitag general type. 

(3) If $Y\not\subset Z_{p}'$, $Y$ is not rationally connected. 
\end{theorem}

Here $Y$ is said to be of \textit{Freitag general type} (\cite{Fr}) 
if there exist holomorphic tensors $\omega_{0}, \cdots, \omega_{N}$ on a smooth projective model of $Y$ such that 
$\omega_{i}=f_{i}\omega_{0}$ for a rational function $f_{i}$ on $Y$ and 
$f_{1}, \cdots, f_{N}$ generate a field of transcendental degree equal to $\dim (Y)$. 
This is a weakened version of being of general type. 

In practice, the calculation of obstructions from the singularities will be intricate. 
We also give a variant of Theorem \ref{thm: intro full} which avoids this obstruction but 
instead imposes a condition on $Y$ (Theorem \ref{thm: criterion full}). 
Apart from the matter of singularities (see \S \ref{ssec: intro method} and \S \ref{ssec: cusp form criterion}), 
Theorem \ref{thm: intro full} or its variant is a generalization of the criterion of Gritsenko-Hulek-Sankaran \cite{GHS1} that 
${\FG}$ itself is of general type if there is a modular form of reflective and cusp slope $<n$, 
which eventually means $Z_n \ne {\FG}$. 
Theorem \ref{thm: intro full} says, in a quantitative manner, that 
the lower a slope can be realized by some modular form, 
the more subvarieties of ${\FG}$ tend to be of general type. 

The most practical case of Theorem \ref{thm: intro full} will be when ${\G}$ is neat. 
In that case, ${\FG}$ is smooth and has no branch divisor, 
and a toroidal compactification of ${\FG}$ may be chosen to be smooth. 
The slope of a modular form is only with regards to the boundary divisor. 
The variety ${\FG}$ itself is of general type when $n\geq 21$ or $n=17$ (\cite{Ma1}). 
A concrete application of Theorem \ref{thm: intro full} yields the following. 
We put $k(n)=4[(n-2)/8]+12$. 
Then $k(n)\sim n/2$ in $n\to \infty$. 
We have $k(n)<n-1$ when $n\geq 22$. 

\begin{theorem}[\S \ref{ssec: proof A}]\label{thm: any neat}
Let $L$ be a lattice of signature $(2, n)$ with $n\geq 22$ and ${\G}$ be a neat subgroup of ${\OL}$. 
Then there exists an algebraic subset $Z\subsetneq {\FG}$ that contains 
\begin{itemize}
\item all nondegenerate subvarieties $Y$ with $\dim (Y) > k(n)$ that are not of general type,  
\item all subvarieties $Y$ with ${\rkY}>k(n)$ that are not of Freitag general type, and   
\item all rationally connected subvarieties $Y$ with ${\rkY} \geq k(n)$. 
\end{itemize}
In particular, the Lang conjecture holds for nondegenerate subvarieties of ${\FG}$ of dimension $> k(n)$.  
\end{theorem}

Since ${\rkY}\geq n - 2{\rm codim}(Y)$, 
a part of this assertion can be simplified when ${\rm codim}(Y) \leq (n-k(n))/2$. 

\begin{corollary}
Let ${\G}$ be as above. 
The locus $Z$ contains all subvarieties of codimension $ < (n-k(n))/2$ that are not of Freitag general type  
and all rationally connected subvarieties of codimension $(n-k(n))/2$.  
\end{corollary}

If we focus on principal congruence subgroups, we can obtain a better result (\S \ref{ssec: pcsg}).  
Let ${\FN}=\mathcal{F}({\GN})$ be the principal congruence modular variety of level $N$ for even $L$ 
as defined in \S \ref{ssec: pcsg}. 
As a consequence of Corollary \ref{thm: pcsg} we obtain the following. 

\begin{corollary}
Let $n\geq 5$ and $N\geq 3$. 
The Lang conjecture holds for nondegenerate subvarieties of ${\FN}$ of dimension $>k(n)/N$. 
\end{corollary}

These give effective partial results to the Lang conjecture for orthogonal modular varieties. 
By a theorem of Brunebarbe \cite{Br1}, 
there exists a subgroup $\Gamma'<\Gamma$ such that all subvarieties of $\mathcal{F}(\Gamma')$ are of general type. 
Theorem \ref{thm: any neat} seems to be a first effective result for arbitrary neat groups.

Cadorel \cite{Ca} proved for general smooth locally symmetric varieties with ${\G}$ neat, that 
there exists a constant $0< C_p \leq 1$ with the property that 
all subvarieties $Y$ of ${\FG}$ of non-general type and $\dim Y=p$ 
are contained in the common zero locus of cusp forms of slope $< nC_p$, 
where $n$ means the canonical weight. 
His method is analytic, based on the construction of singular metrics. 
The Cadorel constant $C_p$ is defined in terms of the curvature of the domain, 
and explicitly computed for balls and Siegel upper half spaces. 
The neat case of Theorem \ref{thm: intro full} means that 
we have a similar and explicit constant, say $C_p'=p/n$, for orthogonal modular varieties 
that applies to nondegenerate subvarieties. 
Our method is totally different, based on vector-valued modular forms and holomorphic tensors. 
Moreover, our method works even when ${\FG}$ has branch divisor, which is not considered in \cite{Ca}, 
and covers the case of slope equal to $p$. 

Historically, there have been two series of approaches in the study of birational type of subvarieties of modular varieties. 
The first one is to use vector-valued modular forms to construct holomorphic tensors on the modular variety. 
This goes back to Freitag \cite{Fr} and was developed mainly for Siegel modular varieties in the 1980's 
(\cite{Fr}, \cite{We}, \cite{Ts} etc). 
The method of the present article belongs to this series, and is close to \cite{Fr} in spirit. 
The second approach comes from the study of hyperbolicity, and is analytic in nature. 
Its origin goes back to the paper of Nadel \cite{Na}, 
and has been developed extensively in recent years (\cite{Ro}, \cite{Br1}, \cite{Ca} etc).  
Both approaches have their respective merits, and are complementary in various aspects.

\subsection{Outline of construction}\label{ssec: intro method}

We now explain our construction in more detail. 
The key construction is an embedding (Proposition \ref{prop: basic embedding}) 
\begin{equation*}
\mathcal{L}^{\otimes pm} \hookrightarrow (\Omega_{{\D}}^p)^{\otimes m} 
\end{equation*}
of the modular line bundle $\mathcal{L}^{\otimes pm}$. 
Usually some representation theory is required to produce such an embedding (cf.~\cite{Fr}, \cite{We}), 
but in the present case this is done by an entirely elementary linear algebra. 
This is the main point of this article. 
This embedding enables to construct the holomorphic tensor $\omega(F)$ from a modular form $F$ of weight $pm$. 
Like the case of pluricanonical forms (\cite{GHS1}), 
there are three types of obstructions to extending $\omega(F)$ over a smooth projective model of ${\FG}$: 
\begin{itemize}
\item \textit{reflective obstruction} from the ramification divisor $R$ of ${\D}\to {\FG}$; 
\item \textit{cusp obstruction} from the boundary divisor $\Delta$ of a toroidal compactification ${\FGcpt}$ of ${\FG}$; 
\item \textit{elliptic obstruction} from certain bad singularities of ${\FGcpt}$.  
\end{itemize}
The reflective/cusp/elliptic slope conditions in Theorem \ref{thm: intro full}  
correspond exactly to these obstructions. 
The reflective and cusp obstructions can be overcome by requiring the modular form $F$ to vanish to order $\geq m$ 
at $R$ and $\Delta$ respectively (Proposition \ref{prop: extension criterion}). 
Several powerful methods are known for constructing such modular forms 
(\cite{Gr}, \cite{Bo}, \cite{GHS1}, \cite{GHS2}). 

On the other hand, when $p<n$, the obstruction from the singularities presents a new problem.  
This is different from the case $p=n$ where ${\FGcpt}$ has canonical singularities when $n\geq 9$ (\cite{GHS1}). 
A simple calculation convinces us that 
general holomorphic tensors cannot be always extended even over canonical singularities 
which arise rather usually (Example \ref{ex: singularity}). 
The reason for this is that the generalized Reid-Tai sum (\cite{We}) is $<1$. 
This explains the necessity to formulate a full version (Theorem \ref{thm: intro full}) incorporating the elliptic slope, 
which amounts to imposing vanishing conditions at the singularities 
that can compensate for the shortage of the generalized Reid-Tai sum. 
In practice, to estimate the elliptic obstruction will require a lot of explicit calculation. 
This is a problem of rather separate nature and will be studied elsewhere.

This article is organized as follows. 
In \S \ref{sec: hol tensor} we construct the embedding 
$\mathcal{L}^{\otimes pm} \hookrightarrow (\Omega_{{\D}}^p)^{\otimes m}$. 
In \S \ref{sec: extension} we establish the extension criterion for $\omega(F)$. 
In \S \ref{sec: proof main} we give applications to the birational type of subvarieties. 

I would like to thank the referee for many valuable and detailed comments.


\section{Modular forms and holomorphic tensors}\label{sec: hol tensor}

In this section we construct the embedding 
$\mathcal{L}^{\otimes pm}\hookrightarrow (\Omega_{{\D}}^{p})^{\otimes m}$ 
of the line bundle $\mathcal{L}^{\otimes pm}$ (Proposition \ref{prop: basic embedding}) 
that plays a central role in this article. 
This is done by studying the ``second" Hodge bundle $\mathcal{E}$. 
In \S \ref{ssec: rank} we define the rank of a subvariety of ${\FG}$.

\subsection{The two automorphic vector bundles}\label{ssec: 2nd Hodge bundle}

Let $L$ be a lattice of signature $(2, n)$. 
Let $Q\subset {\proj}L_{{\C}}$ be the isotropic quadric defined by $(\omega, \omega)=0$. 
Then $Q$ is a homogeneous space of ${\rm O}(L_{{\C}})$. 
The analytic open set of $Q$ defined by the condition $(\omega, \bar{\omega})>0$ consists of two connected components, 
one of which is defined as the Hermitian symmetric domain ${\D}$ attached to $L$. 
The quadric $Q$ is the compact dual of ${\D}$. 
We write ${\OL}$ for the subgroup of the orthogonal group ${\rm O}(L)$ that preserves the component ${\D}$.  
If ${\G}$ is a subgroup of ${\OL}$ of finite index, 
the quotient ${\FG}={\G}\backslash {\D}$ has the structure of a quasi-projective variety of dimension $n$. 

Over ${\D}\subset Q$ we have two fundamental ${\rm O}^{+}(L_{{\R}})$-equivariant vector bundles. 
The first one is the line bundle 
\begin{equation*}
\mathcal{L} = \mathcal{O}_{{\proj}L_{{\C}}}(-1) |_{{\D}}. 
\end{equation*}
By definition we have 
$\mathcal{O}_{{\proj}L_{{\C}}}(-1) \subset L_{{\C}}\otimes \mathcal{O}_{{\proj}L_{{\C}}}$.  
Over ${\D}\subset Q$ this defines the filtration 
\begin{equation*}
0 \subset \mathcal{L} \subset \mathcal{L}^{\perp} \subset L_{{\C}}\otimes \mathcal{O}_{{\D}}, 
\end{equation*}
which reflects the fact that ${\D}$ parametrizes 
polarized Hodge structures of weight $2$ on $L$. 
Our second automorphic vector bundle is 
\begin{equation*}
\mathcal{E} = \mathcal{L}^{\perp}/\mathcal{L}. 
\end{equation*}
By construction $\mathcal{E}$ is equipped with a canonical nondegenerate quadratic form 
induced from that on $L_{{\C}}\otimes \mathcal{O}_{{\D}}$. 
In other words, $\mathcal{E}$ is an orthogonal vector bundle. 
In particular, we have $\mathcal{E}^{\vee}\simeq \mathcal{E}$.

The Hodge line bundle $\mathcal{L}$ may be more familiar 
and it has indeed played a central role in the study of orthogonal modular varieties. 
But the second Hodge bundle $\mathcal{E}$ will also contain a lot of information. 
The connection with holomorphic tensors is provided by the following basic isomorphism. 

\begin{lemma}\label{lem: fundamental isom}
We have $\Omega_{{\D}}^{1} \simeq \mathcal{L}\otimes \mathcal{E}$. 
Taking the exterior power, we obtain 
\begin{equation*}\label{eqn: isom wedge p}
\Omega_{{\D}}^{p} \simeq \mathcal{L}^{\otimes p} \otimes \wedge^{p} \mathcal{E}. 
\end{equation*}
\end{lemma}
 
\begin{proof}
The Euler sequence over ${\proj}L_{{\C}}$ says that 
\begin{equation*}
T{\proj}L_{{\C}} \simeq  
\mathcal{O}_{{\proj}L_{{\C}}}(1) \otimes ( (L_{{\C}}\otimes \mathcal{O}_{{\proj}L_{{\C}}}) / \mathcal{O}_{{\proj}L_{{\C}}}(-1)). 
\end{equation*}
As a sub vector bundle of $T{\proj}L_{{\C}}|_{Q}$, we have 
\begin{equation*}
TQ \simeq \mathcal{O}_{Q}(1) \otimes ( \mathcal{O}_{Q}(-1)^{\perp} / \mathcal{O}_{Q}(-1)). 
\end{equation*}
This shows that $T{\D}\simeq \mathcal{L}^{\vee}\otimes \mathcal{E}$. 
Therefore 
\begin{equation*}
\Omega_{{\D}}^{1} \simeq \mathcal{L}\otimes \mathcal{E}^{\vee} \simeq \mathcal{L}\otimes \mathcal{E} 
\end{equation*}
by the self duality of $\mathcal{E}$. 
\end{proof}

The quadratic form on $\mathcal{E}$ defines an $\mathcal{L}^{\otimes 2}$-valued quadratic form on $\Omega_{{\D}}^{1}$. 
Dually, we have an $\mathcal{L}^{\otimes -2}$-valued quadratic form on $T{\D}$. 
 
\begin{remark}
The reductive part of the stabilizer of a point of $Q$ in ${\rm O}(L_{{\C}})$ 
is isomorphic to ${\C}^{\ast}\times {\rm O}(n, {\C})$. 
In general, an automorphic vector bundle over ${\D}\subset Q$ can be obtained from 
a representation of ${\C}^{\ast}\times {\rm O}(n, {\C})$. 
Here $\mathcal{L}$ and $\mathcal{E}$ correspond to 
the standard representations of ${\C}^{\ast}$ and  ${\rm O}(n, {\C})$ respectively.
\end{remark}

\subsection{Special holomorphic tensors}\label{ssec: special tensor}

By Lemma \ref{lem: fundamental isom}, 
the construction of holomorphic tensors of type $(\Omega^{p})^{\otimes m}$ 
is translated to the construction of vector-valued modular forms with value in 
$\mathcal{L}^{\otimes pm} \otimes (\wedge^{p} \mathcal{E})^{\otimes m}$. 
The automorphic vector bundle $(\wedge^{p} \mathcal{E})^{\otimes m}$ can be decomposed according to 
the decomposition of $(\wedge^{p}{\C}^n)^{\otimes m}$ as a representation of ${\rm O}(n, {\C})$. 
In particular, since the symmetric square ${\rm Sym}^2(\wedge^{p}{\C}^n)$ contains a trivial summand, 
we have the following embedding. 

\begin{proposition}\label{prop: basic embedding}
Let $0 < p \leq n$ and $m>0$ be an even number. 
There exists a natural ${\rm O}^{+}(L_{{\R}})$-equivariant embedding 
\begin{equation}\label{eqn: basic embedding}
\mathcal{L}^{\otimes pm} \hookrightarrow 
({\rm Sym}^{2}\Omega_{{\D}}^{p})^{\otimes m/2} \: \: \subset (\Omega_{{\D}}^{p})^{\otimes m}. 
\end{equation}
Over each point $x\in {\D}$, the image of this embedding is 
\begin{equation*}
\left( \mathcal{L}_{x}^{\otimes 2p} \otimes \sum_{I}e_{I}^2 \right) ^{\otimes m/2}, 
\end{equation*}
where $e_1, \cdots, e_n$ are an orthonormal basis of $\mathcal{E}_{x}$, 
$e_I=e_{i_1}\wedge \cdots \wedge e_{i_p}$ for $I=\{ i_1, \cdots, i_{p} \}$, 
and $I$ runs over all subsets of $\{ 1, \cdots, n \}$ consisting of $p$ elements. 
\end{proposition}

Note that the vector $\sum_{I}e_{I}^{2}$ of $\textrm{Sym}^2(\wedge^{p}\mathcal{E}_{x})$ 
does not depend on the choice of the orthonormal basis $e_1, \cdots, e_n$. 

\begin{proof}
The quadratic form on $\mathcal{E}$ induces a quadratic form on $\wedge^{p}\mathcal{E}$. 
Since this is ${\rm O}^{+}(L_{{\R}})$-invariant, it defines an ${\rm O}^{+}(L_{{\R}})$-equivariant embedding 
$\mathcal{O}_{{\D}}\hookrightarrow {\rm Sym}^2(\wedge^p\mathcal{E})$. 
A twist by $\mathcal{L}^{\otimes 2p}$ then gives 
\begin{equation*}
\mathcal{L}^{\otimes 2p} \hookrightarrow 
\mathcal{L}^{\otimes 2p}\otimes {\rm Sym}^2(\wedge^p\mathcal{E}) \simeq 
{\rm Sym}^2(\Omega_{{\D}}^{p}). 
\end{equation*}
Taking the $m/2$-power, we obtain \eqref{eqn: basic embedding}. 

If $e_1, \cdots, e_n$ are an orthonormal basis of $\mathcal{E}_{x}$, 
then $\{ e_{I} \}_{|I|=p}$ form an orthonormal basis of $\wedge^{p}\mathcal{E}_{x}$. 
This shows that the image of 
$\mathcal{O}_{{\D},x}\hookrightarrow {\rm Sym}^2(\wedge^p\mathcal{E}_{x})$ 
is spanned by $\sum_{I}e_{I}^2$. 
This implies the second assertion. 
\end{proof}

The embedding \eqref{eqn: basic embedding} enables to construct holomorphic tensors on 
a Zariski open set of the modular variety from scalar-valued modular forms. 
Its extendability will be studied in \S \ref{sec: extension}. 
When restricting these holomorphic tensors to subvarieties, 
the following criterion will be used. 

\begin{lemma}\label{lem: restrict}
Let $x\in {\D}$ and $U\subset {\D}$ be an open neighborhood of $x$. 
Let $Y\subset U$ be a complex submanifold of dimension $\geq p$ with $x\in Y$. 
The composition map 
\begin{equation}\label{eqn: restriction}
\mathcal{L}_{x}^{\otimes 2p} \to {\rm Sym}^{2}\Omega_{{\D},x}^{p} \to {\rm Sym}^{2}\Omega_{Y,x}^{p} 
\end{equation} 
is nonzero if and only if 
the restriction of the canonical $\mathcal{L}_{x}^{\otimes -2}$-valued quadratic form on $T_{x}{\D}$ 
to the subspace $T_{x}Y\subset T_{x}{\D}$ has rank $\geq p$.  
\end{lemma}

\begin{proof}
The image of the embedding 
$\mathcal{L}_{x}^{\otimes 2p} \to {\rm Sym}^{2}\Omega_{{\D},x}^{p}$ 
is generated by the quadratic form on $\wedge^{p}T_{x}{\D}$ induced from that on $T_{x}{\D}$. 
Hence the image of \eqref{eqn: restriction} is generated by the restriction of this quadratic form to 
the subspace $\wedge^{p}T_{x}Y \subset \wedge^{p}T_{x}{\D}$. 
This is nonzero, i.e., $\wedge^{p}T_{x}Y$ is not isotropic if and only if 
the restricted quadratic form on $T_{x}Y$ has rank $\geq p$.   
\end{proof}

\subsection{Rank of subvarieties}\label{ssec: rank}

After Lemma \ref{lem: restrict}, we make the following definition. 

\begin{definition}
Let $Y$ be a subvariety of ${\FG}$ not contained in the branch locus of ${\FG}$. 
If $x$ is a smooth point of $Y$ not in the branch locus of ${\FG}$, 
we denote by ${\rm rk}(T_{x}Y)$ the rank of the restriction of 
the quadratic form on $T_{x}{\D}$ 
to the subspace $T_{x}Y$. 
Then we define the \textit{rank} of $Y$ by 
\begin{equation*}
{\rkY} = {\rm rk}(T_{x}Y) 
\end{equation*}
for a general point $x$ of $Y$. 
When ${\rkY}=\dim(Y)$, namely the restricted quadratic form on $T_{x}Y$ is nondegenerate, 
we say that $Y$ is \textit{nondegenerate}. 
\end{definition}

Since the function $x\mapsto {\rm rk}(T_{x}Y)$ is lower semicontinuous, 
${\rkY}$ is the maximum of ${\rm rk}(T_{x}Y)$ for all smooth points $x$ of $Y$ not in the branch locus of ${\FG}$. 
Likewise, if $\mathcal{Y}\to \mathcal{U}$ is a family of subvarieties of ${\FG}$, 
the function $u\mapsto {\rm rk}(\mathcal{Y}_{u})$ over $\mathcal{U}$ is lower semicontinuous. 
In particular, small deformation of a nondegenerate subvariety is again nondegenerate.

\begin{remark}
We can also interpret ${\rkY}$ in terms of the Gauss map. 
Assume for simplicity that ${\G}$ is torsion-free. 
Let $\mathcal{G}=Gr(p, T{\FG})$ be the relative Grassmannian parametrizing 
$p$-dimensional subspaces of $T_{x}{\FG}$, $x\in {\FG}$. 
Let $\mathcal{Z}_{r} \subset \mathcal{G}$ be the degeneracy locus parametrizing 
subspaces of rank $\leq r$. 
If $Y$ is a $p$-dimensional subvariety of ${\FG}$, 
we have the Gauss map $g\colon Y \dashrightarrow \mathcal{G}$ 
sending $x\in Y$ to $(x, T_{x}Y)\in \mathcal{G}$. 
Then ${\rkY}$ is the minimum of $r$ with $g(Y)\subset \mathcal{Z}_{r}$. 
\end{remark}

We summarize some basic properties of ${\rkY}$. 

\begin{lemma}
(1) We have $\dim(Y) \geq {\rkY} \geq \dim(Y) - {\rm codim}(Y)$. 

(2) If $Y_{1}\subset Y_{2}$ with $Y_{1}\not\subset {\rm Sing}(Y_{2})$, 
then ${\rm rk}(Y_{1})\leq {\rm rk}(Y_{2})$. 

(3) Let $\Gamma'$ be a finite-index subgroup of ${\G}$ and 
$Y'\subset \mathcal{F}(\Gamma')$ be an irreducible component of the inverse image of $Y$. 
Then ${\rm rk}(Y')={\rkY}$.  
\end{lemma}

\begin{proof}
Assertion (3) is obvious. 
Assertion (2) follows from the description of ${\rkY}$ as the maximum of ${\rm rk}(T_{x}Y)$. 
The inequality $\dim(Y) \geq {\rkY}$ is apparent. 
The inequality ${\rkY} \geq \dim(Y) - {\rm codim}(Y)$ holds because 
$T_{x}Y\cap (T_{x}Y)^{\perp}$ has nondegenerate pairing with a subspace of $T_{x}{\D}/T_{x}Y$. 
\end{proof}

We compute ${\rkY}$ for a few examples. 
We use the following observation. 

\begin{lemma}\label{lem: rank compute linear}
Let $V$ be a linear subspace of $L_{{\C}}$ with ${\proj}V\cap {\D} \ne \emptyset$ and 
let $x=[{\C}\omega]\in {\proj}V \cap {\D}$. 
Then 
${\rm rk}(T_{x}({\proj}V\cap {\D}))$ 
is equal to the rank of the natural quadratic form on $(\omega^{\perp}\cap V)/{\C}\omega$. 
\end{lemma}

\begin{proof}
This holds because 
$T_{x}({\proj}V\cap {\D}) = 
({\C}\omega)^{\vee} \otimes ((V/{\C}\omega) \cap (\omega^{\perp}/{\C}\omega))$. 
\end{proof}

\begin{example}
Let $M$ be a sublattice of $L$ of signature $(2, \ast)$. 
Then $\mathcal{D}_{M}={\proj}M_{{\C}}\cap {\D}$ is the Hermitian symmetric domain attached to $M$ 
and defines a sub orthogonal modular variety $Y_{M}\subset {\FG}$. 
The quadratic form on $M_{{\C}}$ is nondegenerate. 
Hence the quadratic form on $(\omega^{\perp}\cap M_{{\C}}) /{\C}\omega$ is also nondegenerate 
for $[{\C}\omega]\in \mathcal{D}_{M}$. 
Therefore the subvariety $Y_{M}$ is nondegenerate by Lemma \ref{lem: rank compute linear}. 
\end{example}

\begin{example}
Let again $M$ be a sublattice of $L$ of signature $(2, \ast)$ and 
assume that the quadratic form on $M_{{\Q}}$ underlies a Hermitian form 
over an imaginary quadratic field $F={\Q}(\sqrt{-d})$, $d>0$. 
This means that there exists a ${\Q}$-linear map 
$\xi\colon M_{{\Q}}\to M_{{\Q}}$ such that 
$(\xi l, \xi l')=d(l, l')$ and $\xi^2$ is the multiplication by $-d$. 
Let $M_{{\C}}=V_{+}\oplus V_{-}$ be the eigendecomposition by $\xi$. 
Then $\mathcal{B}_{M}={\proj}V_{+}\cap {\D}$ is a complex ball embedded in $\mathcal{D}_{M}\subset {\D}$ 
of dimension $\dim({\D}_{M})/2$. 
It defines a sub ball quotient $Y_{M,F}\subset Y_{M} \subset {\FG}$. 

The quadratic form on $V_{+}$ is isotropic. 
Indeed, for $v, v'\in V_{+}$, we have 
\begin{equation*}
d(v, v') = (\xi v, \xi v') = (\sqrt{-d}v, \sqrt{-d}v') = -d(v, v') 
\end{equation*}
and hence $(v, v')=0$. 
Therefore the subvariety $Y_{M,F}$ has rank $0$ by Lemma \ref{lem: rank compute linear}. 
In particular, any subvariety contained in $Y_{M,F}$ also has rank $0$. 
\end{example}


\section{Extension criterion}\label{sec: extension}

Let $L$ be a lattice of signature $(2, n)$ with $n\geq 3$ and ${\G}$ be a subgroup of ${\OL}$ of finite index. 
We denote by ${\FGo}\subset {\FG}$ the complement of the branch locus of ${\D}\to {\FG}$. 
A ${\G}$-invariant section of $\mathcal{L}^{\otimes k}$ over ${\D}$ is called a \textit{modular form} of weight $k$ with respect to ${\G}$. 
(Since $n\geq 3$, the holomorphicity at the cusps is automatically satisfied by the Koecher principle.)  
We may also allow a twist by a character of ${\G}$ such as $\det$, 
which we omit for simplicity. 

We fix $0<p\leq n$ and an even number $m>0$. 
Let $F$ be a ${\G}$-modular form of weight $pm$. 
Since $F$ has even weight, it is $\langle {\G}, -{\rm id} \rangle$-modular, 
so we may allow $-{\rm id}\in {\G}$ without losing $F$. 
By Proposition \ref{prop: basic embedding}, 
$F$ gives rise to a holomorphic tensor over ${\FGo}$ of type $(\Omega^{p})^{\otimes m}$, 
which we denote by 
\begin{equation*}
\omega(F) \in H^{0}({\FGo}, (\Omega^{p})^{\otimes m}). 
\end{equation*}
In this section we give a criterion for extendability of $\omega(F)$ 
over the regular locus of a toroidal compactification ${\FGcpt}$ of ${\FG}$, 
and then study extendability over its desingularization.  
The codimension $1$ components of the complement of ${\FGo}$ in ${\FGcpt}$ 
consist of the boundary divisors of ${\FG}\hookrightarrow {\FGcpt}$ and the branch divisors of ${\D}\to{\FG}$. 
The first part can be summarized as follows. 

\begin{proposition}\label{prop: extension criterion}
Assume $-{\rm id}\in {\G}$. 
The holomorphic tensor $\omega(F)$ over ${\FGo}$ extends holomorphically over the regular locus of ${\FGcpt}$ 
if the modular form $F$ satisfies the following. 

(1) $F$ has vanishing order $\geq m$ along every ramification divisor in ${\D}$. 

(2) $F$ has vanishing order $\geq m$ along every boundary divisor. 
\end{proposition}

This is a summary of Lemmas \ref{lem: branch} and \ref{lem: regular boundary}. 
The vanishing order of $F$ along a boundary divisor is defined by \eqref{eqn: vanishing order} in terms of its Fourier expansion. 
The second part, extendability over desingularization, is summarized in Proposition \ref{prop: extension singularity}. 
A certain amount of notation is necessary to state it, 
but it says roughly that $\omega(F)$ can be extended over desingularization 
if $F$ vanishes to certain order at the exceptional divisors over ``bad'' singularities. 

\textit{We assume $-{\rm id}\in {\G}$ in \S \ref{ssec: branch}, \S \ref{ssec: boundary}, \S \ref{ssec: singularity}, 
but do not so in \S \ref{ssec: toroidal}.} 
The assumption $-{\rm id}\in{\G}$ is made for simplicity of exposition. 
It is possible to work without this assumption, 
but then we need to switch from ${\G}$ to $\langle {\G}, -{\rm id} \rangle$ at some points, 
and also pay attention to the irregular cusps. 
We tried to keep these complications as minimal as possible (only inside \S \ref{ssec: toroidal}). 
In \S \ref{sec: proof main} we work with general ${\G}$ not necessarily containing $-{\rm id}$. 
There we will apply Propositions \ref{prop: extension criterion} and \ref{prop: extension singularity} 
to $\langle {\G}, -{\rm id}\rangle$ (see the proof of Theorem \ref{thm: full}). 
Thus the assumption $-{\rm id}\in {\G}$ here is sufficient for the application in \S \ref{sec: proof main}.

\subsection{Branch divisor}\label{ssec: branch}

Let $B\subset {\FG}$ be the branch divisor of $\pi\colon {\D}\to {\FG}$ and 
$B=\sum_{i}B_{i}$ be its irreducible decomposition. 
We give a criterion for extendability of $\omega(F)$ over a general point of $B_i$. 
We take an irreducible component $R_i\subset {\D}$ of the ramification divisor that lies over $B_i$. 
By \cite{GHS1} Corollary 2.13, $R_i$ is defined by a reflection in ${\G}$, 
and the ramification index is $2$. 
Let $\nu_{R_i}(F)$ be the vanishing order of $F$ along $R_i$. 

\begin{lemma}\label{lem: branch}
The holomorphic tensor $\omega(F)$ extends holomorphically over a general point of $B_i$ if $\nu_{R_i}(F)\geq m$. 
\end{lemma}

\begin{proof}
Let $x$ be a general point of $R_i$ and $\pi(x) \in B_i$ be its image in ${\FG}$. 
We may take local coordinates $z_1, \cdots, z_n$ of ${\D}$ around $x$ 
and local coordinates $w, z_2, \cdots , z_n$ of ${\FG}$ around $\pi(x)$ 
such that ${\D}\to {\FG}$ is given by $w=z_{1}^{2}$ around $x$. 
Locally, $R_i$ is defined by $z_1=0$ and $B_i$ is defined by $w=0$. 
We write $dz_{I}=dz_{i_1}\wedge \cdots \wedge dz_{i_p}$ for $I=\{ i_1, \cdots , i_p \}$. 
We identify $F$ with a local holomorphic function by taking a local frame of $\mathcal{L}$. 
Then, locally around $x$, we can write 
\begin{equation*}
\pi^{\ast}\omega(F) = F\cdot \sum_{I_1,\cdots,I_m}a_{I_1,\cdots,I_m}dz_{I_1}\otimes \cdots \otimes dz_{I_m}, 
\end{equation*}
with $a_{I_1,\cdots,I_m}$ a local holomorphic function. 
The point is that $dz_1$ appears at most $m$ times in $dz_{I_1}\otimes \cdots \otimes dz_{I_m}$. 
We substitute $2dz_1=z_{1}^{-1}dw$. 
Then, locally around $\pi(x)\in {\FG}$, we can write 
\begin{equation*}
\omega(F) = F \cdot \sum_{I_1,\cdots,I_m} z_{1}^{-\mu} \cdot a_{I_1,\cdots,I_m} \cdot \omega_{I_1,\cdots,I_m}, 
\end{equation*}
where $\mu=\mu(I_1, \cdots , I_m)\leq m$ and  
$\omega_{I_1,\cdots,I_m}$ is a tensor product of $dw, dz_{2}, \cdots ,  dz_{n}$ 
obtained from $dz_{I_1}\otimes \cdots \otimes dz_{I_m}$
by replacing $dz_1$ with $2^{-1}dw$. 
Since $\nu_{R_i}(F)\geq m \geq \mu$, 
we find that $\omega(F)$ extends holomorphically over $B_{i}$ around $\pi(x)$. 
\end{proof}

\subsection{Toroidal compactification}\label{ssec: toroidal}

In this subsection we recall toroidal compactifications of ${\FG}$ and Fourier expansions of modular forms. 
This will be used in \S \ref{ssec: boundary}, \S \ref{ssec: singularity} and \S \ref{sec: proof main}. 
We refer to \cite{AMRT}, \cite{GHS1}, \cite{Ma2} for more detail. 
The Baily-Borel compactification of ${\FG}$ has $0$-dimensional and $1$-dimensional cusps, 
and a toroidal compactification is obtained from partial compactifications over them (\cite{AMRT}). 
We do not assume $-{\rm id}\in {\G}$ in this \S \ref{ssec: toroidal}.

\subsubsection{$0$-dimensional cusps}

A $0$-dimensional cusp of ${\FG}$ corresponds to a ${\G}$-equivalence class of 
rank $1$ primitive isotropic sublattices $I$ of $L$. 
We write $L(I)=(I^{\perp}/I)\otimes I$, which is canonically equipped with a hyperbolic quadratic form. 
The projection 
$Q\dashrightarrow {\proj}(L/I)_{{\C}}$ 
from the boundary point $[ I_{{\C}} ]\in Q$ 
embeds the domain ${\D}\subset Q$ into the affine space ${\proj}(L/I)_{{\C}}-{\proj}(I^{\perp}/I)_{{\C}}$. 
A choice of a rank $1$ sublattice $I'\subset L$ with $(I, I')\ne 0$ 
determines a base point of this affine space and thus its isomorphism with $L(I)_{{\C}}$. 
Then the image of ${\D}$ in $L(I)_{{\C}}$ is 
the tube domain $\mathcal{D}_{I}\subset L(I)_{{\C}}$ defined by the condition ${\rm Im}(Z) \in \mathcal{C}_{I}$, 
where $\mathcal{C}_{I}$ is the positive cone of $L(I)_{{\R}}$. 

Let $\Gamma(I)_{{\Q}}$ be the stabilizer of $I$ in ${\rm O}^{+}(L_{{\Q}})$ and 
$U(I)_{{\Q}}$ be the unipotent part of $\Gamma(I)_{{\Q}}$. 
Then $U(I)_{{\Q}}$ is canonically isomorphic to $L(I)_{{\Q}}$, 
consisting of the Eichler transvections. 
The action of $U(I)_{{\Q}}$ on ${\D}$ is identified with 
the translation by $L(I)_{{\Q}}$ on $\mathcal{D}_{I}\subset L(I)_{{\C}}$. 
We put $\Gamma(I)_{{\Z}}=\Gamma(I)_{{\Q}}\cap \Gamma$ and ${\UIZ}=U(I)_{{\Q}}\cap \Gamma$. 
Then ${\UIZ}$ is a lattice on the linear space $U(I)_{{\Q}}\simeq L(I)_{{\Q}}$ 
(with the quadratic form not necessarily ${\Z}$-valued).  
The quotient ${\D}/{\UIZ}$ is an open set of the algebraic torus $T(I)=L(I)_{{\C}}/{\UIZ}$. 

Let $\mathcal{C}_{I}^{\ast}\subset U(I)_{{\R}}\simeq L(I)_{{\R}}$ 
be the union of $\mathcal{C}_{I}$ and rational isotropic rays in the closure of $\mathcal{C}_{I}$. 
We take a $\Gamma(I)_{{\Z}}$-admissible cone decomposition 
$\Sigma_{I}=(\sigma_{\alpha})$ of $\mathcal{C}_{I}^{\ast}$. 
This defines a torus embedding $T(I)\hookrightarrow T(I)^{\Sigma_{I}}$. 
The partial compactification $({\D}/{\UIZ})^{\Sigma_{I}}$ of ${\D}/{\UIZ}$ is defined as 
the interior of the closure of ${\D}/{\UIZ}$ in $T(I)^{\Sigma_{I}}$. 
By construction, the boundary divisors $D_{\sigma}$ of $({\D}/{\UIZ})^{\Sigma_{I}}$ 
correspond to the rays $\sigma={\R}_{\geq 0}v\subset U(I)_{{\R}}$ in $\Sigma_{I}$. 
We call a ray $\sigma$ \textit{regular} if 
${\UIZ}\cap \sigma = U(I)_{{\Q}}\cap \langle {\G}, -{\rm id } \rangle\cap \sigma$ 
in $U(I)_{{\R}}$, 
and \textit{irregular} otherwise (\cite{Ma2}). 
The cusp $I$ is called \textit{regular} if ${\UIZ}=U(I)_{{\Q}}\cap \langle {\G}, -{\rm id } \rangle$, 
and \textit{irregular} otherwise. 
Irregular rays can arise only when $I$ is irregular. 
The group ${\G}$ has no irregular cusp when $-{\rm id}\in {\G}$ or when ${\G}$ is neat.

\subsubsection{Fourier expansion}\label{sssec: Fourier}

Let $F$ be a ${\G}$-modular form of weight $k$. 
A choice of a generator $l_{I}$ of $I$ determines a natural frame $s_{I}$ of the line bundle $\mathcal{L}$ 
by the condition $(l_{I}, s_{I})=1$. 
Via the frame $s_{I}^{\otimes k}$ of $\mathcal{L}^{\otimes k}$,  
$F$ is identified with a holomorphic function on ${\D}\simeq \mathcal{D}_{I}$, 
again denoted by $F$. 
Since the function $F$ is invariant under the translation by the lattice ${\UIZ}$, it admits a Fourier expansion: 
\begin{equation*}
F(Z) = \sum_{l\in U(I)_{{\Z}}^{\vee}} a(l) q^{l}, \quad q^{l}={\rm exp}(2\pi i(l, Z)),  
\end{equation*}
where $Z\in \mathcal{D}_{I}$ 
and $U(I)_{{\Z}}^{\vee}$ is the dual lattice of ${\UIZ}$. 
By the holomorphicity at the cusp, the vectors $l$ range only over $U(I)_{{\Z}}^{\vee}\cap \mathcal{C}_{I}^{\ast}$. 
The modular form $F$ is called a \textit{cusp form} if 
$a(l)=0$ for all isotropic vectors $l\in U(I)_{{\Z}}^{\vee}\cap \mathcal{C}_{I}^{\ast}$ 
at all $0$-dimensional cusps $I$. 

Let $\sigma={\R}_{\geq0}v$ be a ray in $\Sigma_{I}$. 
We take the generator $v$ to be a primitive vector of ${\UIZ}$. 
The vanishing order of $F$ at $\sigma$ is defined by 
\begin{equation}\label{eqn: vanishing order}
\nu_{\sigma}(F) = \min \{ \: (l, v) \: | \: a(l)\ne 0 \: \}. 
\end{equation}
This is equal to the vanishing order of $F$ along the divisor $D_{\sigma}$ 
as a section of $\mathcal{L}^{\otimes k}$ extended over $({\D}/{\UIZ})^{\Sigma_{I}}$ via $s_{I}^{\otimes k}$. 
$F$ is a cusp form if and only if $\nu_{\sigma}(F)>0$ 
for all rays $\sigma$ at all $0$-dimensional cusps $I$.  
We also define the geometric vanishing order of $F$ at $\sigma$ by  
\begin{equation}\label{eqn: geom vanish order}
\nu_{\sigma, geom}(F) = 
\begin{cases}
\nu_{\sigma}(F), & \sigma : \: \textrm{regular} \\ 
\nu_{\sigma}(F)/2, & \sigma : \: \textrm{irregular}
\end{cases}
\end{equation}
This may be a half-integer when $\sigma$ is irregular. 
In any case, $2\nu_{\sigma, geom}(F)$ is equal to 
the vanishing order of $F^{2}$ at $\sigma$ as a $\langle {\G}, -{\rm id} \rangle$-modular form.

\subsubsection{$1$-dimensional cusps}

A $1$-dimensional cusp of ${\FG}$ corresponds to 
a ${\G}$-equivalence class of rank $2$ primitive isotropic sublattices $J$ of $L$. 
A full explanation of the partial compactification over $J$ requires a rather long description. 
Instead of that, we just collect a few properties that will be directly relevant to the rest of this article. 
See \cite{GHS1}, \cite{Ma2} for more detail. 

\begin{itemize}
\item ${\UJZ}$, the integral part of the center of the unipotent radical of the stabilizer of $J$, has rank $1$. 
\item The partial compactification $\overline{{\D}/{\UJZ}}$ over $J$ is canonical, requiring no choice. 
It is obtained by filling the origin of punctured disc in a family. 
The boundary divisor $D_J$ is irreducible. 
\item If we choose a rank $1$ primitive sublattice $I\subset J$, we have an etale gluing map 
$\overline{{\D}/{\UJZ}} \to ({\D}/{\UIZ})^{\Sigma_{I}}$. 
The image of $D_J$ is the boundary divisor $D_{\sigma_J}$ of $({\D}/{\UIZ})^{\Sigma_{I}}$ 
corresponding to the isotropic ray 
$\sigma_J = ((J/I)_{{\R}}\otimes I_{{\R}})_{\geq 0}$ 
in $L(I)_{{\R}}$.  
\end{itemize}

The gluing map 
$\overline{{\D}/{\UJZ}} \to ({\D}/{\UIZ})^{\Sigma_{I}}$
reduces some descriptions concerning $D_J$ to those concerning $D_{\sigma_J}$ 
(e.g., Fourier-Jacobi expansion, vanishing order, (ir)regularity). 
In this sense, the $1$-dimensional cusp $J$ is reduced, to some extent, to the isotropic ray $\sigma_J$ at 
the adjacent $0$-dimensional cusp $I$.

\subsubsection{Toroidal compactification}

A toroidal compactification ${\FGcpt}$ of ${\FG}$ is defined by choosing a collection of admissible fans $\Sigma=(\Sigma_I)$, 
one for each ${\G}$-equivalence class of rank $1$ primitive isotropic sublattices $I$ of $L$ independently. 
No choice is required for rank $2$ isotropic sublattices. 
The space ${\FGcpt}$ is obtained by 
gluing ${\FG}$ and natural quotients of the partial compactifications 
$({\D}/{\UIZ})^{\Sigma_{I}}$, $\overline{{\D}/{\UJZ}}$ 
around the boundary over all cusps $I, J$ (\cite{AMRT}). 
Then ${\FGcpt}$ is a compact Moishezon space containing ${\FG}$ as a Zariski open set 
and having a morphism to the Baily-Borel compactification. 
We may choose $\Sigma$ so that ${\FGcpt}$ is projective. 

We have a projection 
$\pi_{I}\colon ({\D}/{\UIZ})^{\Sigma_{I}} \to {\FGcpt}$ 
for each $0$-dimensional cusp $I$. 
A non-isotropic ray $\sigma\in \Sigma_I$ is regular if and only if 
$\pi_{I}$ is unramified at a general point of the corresponding boundary divisor $D_{\sigma}$. 
When $\sigma$ is irregular, $\pi_{I}$ is doubly ramified at $D_{\sigma}$. 
Similarly, we have a projection 
$\pi_{J}\colon \overline{{\D}/{\UJZ}} \to {\FGcpt}$ 
for each $1$-dimensional cusp $J$. 
If we choose $I\subset J$, 
$\pi_{J}$ is unramified at general points of $D_J$ if and only if 
$\pi_{I}$ is so at $D_{\sigma_J}$.  
This in turn is equivalent to the isotropic ray $\sigma_{J}\in \Sigma_{I}$ being regular.  

\subsection{Boundary divisor}\label{ssec: boundary}

We go back to our analysis of holomorphic tensors. 
Let $F$ and $\omega(F)$ be as in the beginning of this \S \ref{sec: extension}. 
Recall that we assume $-{\rm id}\in {\G}$ in this \S \ref{ssec: boundary}. 
We take a toroidal compactification ${\FGcpt}$ of ${\FG}$. 
We give a criterion for extendability of $\omega(F)$ over general points of the boundary divisors of ${\FGcpt}$. 
Since $-{\rm id}\in {\G}$, ${\FGcpt}$ has no irregular boundary divisor. 

\begin{lemma}\label{lem: regular boundary}
Let $\Delta_{\sigma}$ be an irreducible component of the boundary divisor of ${\FGcpt}$ 
and $\sigma \in \Sigma_{I}$ be a corresponding ray.  
The holomorphic tensor $\omega(F)$ extends holomorphically over a general point of $\Delta_{\sigma}$ 
if $F$ has vanishing order $\nu_{\sigma}(F) \geq m$ at $\sigma$. 
\end{lemma}

\begin{proof}
Let $D_{\sigma}$ be the boundary divisor of $({\D}/{\UIZ})^{\Sigma_{I}}$ corresponding to the ray $\sigma$. 
The projection $({\D}/{\UIZ})^{\Sigma_{I}}\to {\FGcpt}$ maps $D_{\sigma}$ to $\Delta_{\sigma}$, 
and is unramified at a general point of $D_{\sigma}$ by the regularity of $\sigma$. 
Therefore it is sufficient to show that $\omega(F)$ as a holomorphic tensor over ${\D}/{\UIZ}$ 
extends holomorphically over a general point of $D_{\sigma}$. 

We write $\sigma={\R}_{\geq0}v$ with $v\in {\UIZ}$ primitive. 
We extend $v$ to a basis $v_1=v, v_2, \cdots, v_n$ of ${\UIZ}$. 
Its dual basis defines flat coordinates $z_1, \cdots, z_n$ on $\mathcal{D}_{I}\subset L(I)_{{\C}}$. 
Then $q={\rm exp}(2\pi i z_{1}), z_2, \cdots, z_n$ define a local chart of 
$({\D}/{\UIZ})^{\Sigma_{I}}$ around a general point of $D_{\sigma}$, 
with $D_{\sigma}$ defined by $q=0$. 
Recall from \S \ref{ssec: 2nd Hodge bundle} that 
we have a $\mathcal{L}^{\otimes 2}$-valued quadratic form on $\Omega_{{\D}}^{1}$. 

\begin{claim}\label{claim1}
The value of the pairing $(dz_i, dz_j)$ relative to the frame $s_{I}^{\otimes 2}$ of $\mathcal{L}^{\otimes 2}$ 
is constant over ${\D}$. 
\end{claim} 

\begin{proof}
Both $(dz_i, dz_j)$ and $s_{I}^{\otimes 2}$ are sections of $\mathcal{O}(-2)$ defined over the Zariski open set 
$Q - Q\cap {\proj}I^{\perp}_{{\C}} \simeq {\proj}(L/I)_{{\C}}-{\proj}(I^{\perp}/I)_{{\C}}$
of $Q$.  
Since both sections are $U(I)_{{\C}}$-invariant, 
the ratio $(dz_i, dz_j)/s_{I}^{\otimes 2}$ is a $U(I)_{{\C}}$-invariant function on 
${\proj}(L/I)_{{\C}}-{\proj}(I^{\perp}/I)_{{\C}}$. 
Since $U(I)_{{\C}}$ acts on the affine space ${\proj}(L/I)_{{\C}}-{\proj}(I^{\perp}/I)_{{\C}}$ 
transitively (and freely) by translation, 
this ratio is constant. 
\end{proof}

We return to the proof of Lemma \ref{lem: regular boundary}. 
By Claim \ref{claim1}, there exists a transformation matrix $(a_{ij})\in {\rm GL}_{n}({\C})$, 
constant over ${\D}$, such that the new frame  
$e_i=\sum_{j}a_{ij}dz_{j}$ of $\Omega_{{\D}}^1$ 
satisfies $(e_i, e_j)=\delta_{i,j}s_{I}^{\otimes 2}$. 
In other words, $(e_i\otimes s_{I}^{-1})_{i}$ is an orthonormal frame of $\mathcal{E}$. 
It follows that $\omega(F)$ can be written as
\begin{eqnarray*}
\omega(F) 
& = & 
F \cdot ( \sum_{|I|=p}e_{I}^{2})^{\otimes m/2} \\ 
& = & 
F \cdot \sum_{I_1,\cdots, I_m} a_{I_1,\cdots, I_m} dz_{I_1} \otimes \cdots \otimes dz_{I_m}, 
\end{eqnarray*}
where $a_{I_1,\cdots, I_m}$ are \textit{constant}. 
Here $F$ is identified with a function on the tube domain $\mathcal{D}_{I}$ 
via the frame $s_{I}^{\otimes pm}$ of $\mathcal{L}^{\otimes pm}$. 

We substitute $2\pi i dz_{1} = dq/q$. 
Since $dz_1$ appears at most $m$ times in $dz_{I_1} \otimes \cdots \otimes dz_{I_m}$, 
we have 
\begin{equation}\label{eqn: Fourier expansion omegaF regular}
\omega(F) = F \cdot \sum_{I_1,\cdots, I_m} q^{-\mu} \cdot a_{I_1,\cdots, I_m} \cdot  \omega_{I_1,\cdots, I_m}, 
\end{equation}
where $\mu=\mu(I_1, \cdots, I_m) \leq m$ and 
$\omega_{I_1,\cdots, I_m}$ is a tensor product of $dq, dz_2, \cdots, dz_n$ 
obtained from $dz_{I_1}\otimes \cdots \otimes dz_{I_m}$ by replacing $dz_{1}$ with $(2\pi i)^{-1}dq$. 
Since 
$\nu_{\sigma}(F)\geq m\geq \mu$ by assumption, 
we see that the function $F\cdot q^{-\mu}$ on ${\D}/{\UIZ}$ 
extends holomorphically over $D_{\sigma}$ for each $(I_1, \cdots, I_m)$. 
Therefore $\omega(F)$ extends holomorphically over $D_{\sigma}$. 
\end{proof}

\begin{remark}
Lemma \ref{lem: regular boundary} (and hence Proposition \ref{prop: extension criterion}) 
hold true for general ${\G}$ not necessarily containing $-{\rm id}$ 
if we instead require $\nu_{\sigma}(F)\geq m$ at regular $\sigma$ and 
$\nu_{\sigma}(F)\geq 2m$ at irregular $\sigma$. 
The effect of assuming $-{\rm id}\in {\G}$ is to replace the given ${\G}$ by $\langle {\G}, -{\rm id} \rangle$ 
when checking the criterion. 
\end{remark}

\subsection{Reid-Tai-Weissauer criterion}\label{ssec: RTW}

In this subsection, which contains preparatory material for \S \ref{ssec: singularity}, 
we recall the Reid-Tai-Weissauer criterion (\cite{We}) for 
extendability of holomorphic tensors over quotient singularities. 
We also discuss extendability in the case when the RTW criterion cannot be applied. 
This is necessary as there are many examples on orthogonal modular varieties  
where the RTW criterion fails (Example \ref{ex: singularity}). 

Let $V$ be a ${\C}$-linear space of dimension $n$ and 
$\gamma$ be an element of ${\rm GL}(V)$ of finite order which is not a quasi-reflection. 
We write the eigenvalues of $\gamma$ as 
$\exp(2\pi i \alpha_1), \cdots, \exp(2\pi i \alpha_n)$ 
with $0\leq \alpha_i <1$. 
Let $0 < p \leq n$. 
We define the $p$-th Reid-Tai sum of $\gamma$ in two steps. 
When $\langle \gamma \rangle$ contains no quasi-reflection, 
we set 
\begin{equation*}
RT_p(\gamma, V) = \min_{|I|=p} (\alpha_{i_1}+\cdots + \alpha_{i_p}), 
\end{equation*}
where $I=\{ i_1, \cdots , i_p \}$ ranges over all subsets of $\{ 1, \cdots, n \}$ 
consisting of $p$ elements. 
When $p=n$, $RT_n(\gamma, V)$ is the usual Reid-Tai sum (\cite{Re}, \cite{Ta}). 

In general, let 
$\langle \gamma^{d} \rangle$ be the largest subgroup of $\langle \gamma \rangle$ 
which acts on $V$ as quasi-reflection. 
Then $V/\langle \gamma^{d} \rangle$ is naturally isomorphic to a linear space, 
and the quotient group $\langle \gamma \rangle / \langle \gamma^{d} \rangle$ 
acts on $V/\langle \gamma^{d} \rangle$ linearly without quasi-reflection. 
In this situation we define 
\begin{equation*}
RT_{p}(\gamma, V) = RT_p(\bar{\gamma}, V/\langle \gamma^{d} \rangle ), 
\end{equation*}
where $\bar{\gamma}$ is the image of $\gamma$ in $\langle \gamma \rangle / \langle \gamma^{d} \rangle$. 
Explicitly, this is written as follows. 
We may assume that $d\alpha_{i}\in {\Z}$ for $i<n$. 
Let $d'={\rm ord}(\gamma)/d$. 
We write 
$\alpha_{i}'=\alpha_{i}$ for $i<n$ and 
$\alpha_{n}'=d'\alpha_{n}-[d'\alpha_{n}]$, 
the fractional part of $d'\alpha_{n}$. 
Then $\exp(2\pi i \alpha_1'), \cdots, \exp(2\pi i \alpha_n')$ are the eigenvalues of $\bar{\gamma}$ on $V/\langle \gamma^{d} \rangle$. 
So we have 
\begin{equation*}
RT_p(\gamma, V) = \min_{|I|=p} (\alpha_{i_1}'+\cdots + \alpha_{i_p}').  
\end{equation*}
When $p=n$, this modified Reid-Tai sum is considered in \cite{GHS1}. 

The Reid-Tai-Weissauer criterion is the following. 

\begin{theorem}[\cite{We}]\label{thm: RTW}
Let $G$ be a finite subgroup of ${\rm GL}(V)$. 
If $RT_{p}(\gamma, V)\geq 1$ for all $\gamma\ne 1 \in G$ which does not act as a quasi-reflection, 
then every holomorphic tensor of type $(\Omega^{p})^{\otimes m}$ 
over the regular locus of $V/G$ extends holomorphically over a desingularization of $V/G$. 
\end{theorem}

\begin{proof}
When $G$ contains no quasi-reflection, this is proved in \cite{We} Lemma 4. 
In the general case, by the cyclic reduction (\cite{Ta} Proposition 3.1), 
it suffices to show that every holomorphic tensor of type $(\Omega^{p})^{\otimes m}$ 
on the regular locus of $V/\langle \gamma \rangle$ 
extends holomorphically over desingularization of $V/\langle \gamma \rangle$ for any $\gamma \in G$. 
Since 
$V/\langle \gamma \rangle = 
(V/\langle \gamma^{d} \rangle) / (\langle \gamma \rangle / \langle \gamma^{d} \rangle)$, 
we can apply the version of \cite{We} to 
the quasi-reflection-free action of $\langle \gamma \rangle / \langle \gamma^{d} \rangle$ 
on $V/\langle \gamma^{d} \rangle$. 
\end{proof}
 
Theorem \ref{thm: RTW} holds in fact at the level of germs. 
When $p=n$, Theorem \ref{thm: RTW} is the Reid-Tai criterion for canonical singularities (\cite{Re}, \cite{Ta}). 
When $p=n-1$, Theorem \ref{thm: RTW} is also presented in \cite{Ts}. 
In general, as $p$ gets smaller, $RT_{p}$ decreases, 
and it gets harder to achieve $RT_{p}\geq 1$. 
Even for some canonical singularities and at $p=n-1$, 
we may have $RT_{n-1}<1$. 

However, even when $RT_{p}<1$, 
we can extend some holomorphic tensors if we impose sufficient vanishing conditions at the singularities. 
In what follows, we assume that $G$ is cyclic: 
this is justified by the cyclic reduction (\cite{Ta}). 
Then $V/G$ is a toric variety in a natural way (\cite{Ta}). 
We take its toric resolution $\widetilde{V/G}$. 
The same fan also defines a blow-up $\tilde{V}$ of $V$ acted on by $G$ such that $\tilde{V}/G=\widetilde{V/G}$. 
Let $E$ be an irreducible component of the exceptional divisor of $\widetilde{V/G}$ and 
$D$ be an irreducible component of the exceptional divisor of $\tilde{V}$ that lies over $E$ 
with ramification index $r$.

Let $\omega$ be a (local) holomorphic tensor of type $(\Omega^{p})^{\otimes m}$ over $V$, 
and $F$ be a (local) holomorphic function over $V$ such that $F\omega$ is $G$-invariant. 
Then $F\omega$ descends to a (local) holomorphic tensor of type $(\Omega^{p})^{\otimes m}$ over 
the complement of the branch locus of $V/G$, 
which we denote by $\omega(F)$. 
Let $\nu_{D}(F)$ be the vanishing order of $F$ along $D$. 
We put 
\begin{equation}\label{eqn: vanish order excep div}
\nu_{E}(F)=\nu_{D}(F)/(r-1). 
\end{equation} 

\begin{lemma}\label{lem: vanish order vs RTp}
If $\nu_{E}(F)\geq m$, $\omega(F)$ extends holomorphically over $E$. 
\end{lemma}

\begin{proof}
This can be seen by the same calculation as in the proof of Lemma \ref{lem: branch}. 
\end{proof}

If we look at the structure of $\tilde{V}\to \tilde{V}/G$ as in \cite{Ta}, 
we can also derive a bound incorporating $RT_{p}$. 
Assume $G$ contains no quasi-reflection. 
Then $\omega(F)$ extends over $E$ if 
$\nu_{D}(F)/r\geq m(1-RT_{p}(\gamma, V))$ 
for any $\gamma \ne {\rm id} \in G$. 
When $RT_{p}(\gamma, V)=0$ for some $\gamma$, 
the bound in Lemma \ref{lem: vanish order vs RTp} is better. 
We could also improve the bound by a calculation using an explicit toric resolution. 
We use the coarse bound in Lemma \ref{lem: vanish order vs RTp} for simplicity in later sections.

\subsection{Extension over bad singularities}\label{ssec: singularity}

We go back to our analysis of holomorphic tensors over the modular variety ${\FG}$. 
Recall that we assume $-{\rm id}\in {\G}$ in this \S \ref{ssec: singularity}. 
We take a toroidal compactification ${\FGcpt}$ 
such that each fan $\Sigma_I$ is nonsingular with respect to the lattice ${\UIZ}$. 
Every singular point $P$ of ${\FGcpt}$ has a natural model $V_x/G_x$ as a quotient singularity: 
\begin{itemize}
\item For singularities in the interior ${\FG}$, 
we have $V_x=T_x{\D}$ 
where $x\in {\D}$ and $G_x$ is the stabilizer of $x$ in ${\G}$. 
\item For singularities over a $0$-dimensional cusp $I$, 
we have $V_x=T_xT(I)^{\Sigma_{I}}$ 
where $x$ is a boundary point of the torus embedding $T(I)^{\Sigma_I}$ 
and $G_x$ is the stabilizer of $x$ in $\Gamma(I)_{{\Z}}/{\UIZ}$.  
\item We have a similar description for singularities over a $1$-dimensional cusp $J$. 
In fact, $V_x$ can be identified with $T_xT(I)^{\Sigma_{I}}$ 
for an adjacent $0$-dimensional cusp $I\subset J$ by the gluing. 
\end{itemize}

We define 
${\rm Sing}({\FGcpt})_{p}$ 
to be the locus of singular points $P$ where $RT_{p}(\gamma, V_x)<1$ for some $\gamma\in G_x$ 
which is not a quasi-reflection. 
This is a Zariski closed subset of ${\FGcpt}$. 
We have the filtration  
\begin{equation*}
{\rm Sing}({\FGcpt}) = {\rm Sing}({\FGcpt})_{1} \supset \cdots \supset {\rm Sing}({\FGcpt})_{n}. 
\end{equation*}
By \cite{GHS1}, we have ${\rm Sing}({\FGcpt})_{n}=\emptyset$ when $n\geq 9$.

Now let $F$ and $\omega(F)$ be as in the beginning of \S \ref{sec: extension}. 
We assume that $F$ satisfies the condition in Proposition \ref{prop: extension criterion} 
so that $\omega(F)$ extends over the regular locus of ${\FGcpt}$. 
Let $P$ be a singular point contained in ${\rm Sing}({\FGcpt})_{p}$ with quotient model $V_x/G_x$, 
and $G<G_x$ be a cyclic subgroup. 
We take a toric resolution of $V_{x}/G$ and 
let $E$ be an irreducible component of its exceptional divisor. 
We give an extension criterion for $\omega(F)$ over $E$. 

We first consider the case $P\in {\FG}$. 
We identify $F$ with a local function on $V_{x}$ by taking a local frame of $\mathcal{L}$. 
Let $\nu_E(F)$ be as defined in \eqref{eqn: vanish order excep div}. 
By Lemma \ref{lem: vanish order vs RTp}, $\omega(F)$ extends holomorphically over $E$ 
if $\nu_{E}(F)\geq m$.

Next we consider the case when $P$ is a boundary point. 
Thus $x$ is a boundary point of $({\D}/{\UIZ})^{\Sigma_{I}}$ for a $0$-dimensional cusp $I$. 
Let ${\rm orb}(\sigma)$ be the boundary stratum to which $x$ belongs (\cite{AMRT}), 
where $\sigma$ is a cone in $\Sigma_{I}$. 
We take a basis $v_1, \cdots, v_n$ of ${\UIZ}$ such that $\sigma$ is spanned by $v_1, \cdots, v_d$. 
Let $z_1, \cdots, z_n$ be the flat coordinates on the tube domain $\mathcal{D}_{I}\subset L(I)_{{\C}}$ 
defined by the dual basis of $v_1, \cdots, v_n$ and let $q_k={\rm exp}(2\pi i z_k)$. 
Then $q_1, \cdots, q_d, z_{d+1}, \cdots, z_n$ define a local chart around $x$ 
in which the boundary is defined by $q_1\cdots q_d=0$ and 
${\rm orb}(\sigma)$ is defined by $q_1= \cdots = q_d=0$. 
We identify $F$ with a function on ${\D}/{\UIZ}$ as in \S \ref{sssec: Fourier} and consider the function 
$F_x=F\cdot \prod_{i=1}^{d}q_{i}^{-m}$ 
around $x$. 
By our assumption on $F$, $F_{x}$ is holomorphic at the boundary. 
By the proof of Lemma \ref{lem: regular boundary}, 
the pullback of $\omega(F)$ to $V_x$ is written as $F_{x} \omega_{x}$ for 
a local holomorphic tensor $\omega_{x}$ over $V_x$. 
By applying Lemma \ref{lem: vanish order vs RTp} to $F_{x} \omega_{x}$, 
we find that $\omega(F)$ extends over $E$ if $\nu_{E}(F_{x})\geq m$.

To summarize, let 
\begin{equation*}\label{eqn: elliptic vanishing order}
\nu_{E}(F) = \min_{P,G,E_{i}} ( \nu_{E_i}(F \, \textrm{or} \, F_{x})), 
\end{equation*}
where 
$P$ ranges over points in ${\rm Sing}({\FGcpt})_{p}$ with quotient model $V_x/G_x$,  
$G$ ranges over cyclic subgroups of $G_{x}$, 
$E_{i}$ ranges over irreducible components of the exceptional divisor of a toric resolution of $V_x/G$, 
and $F_{x}$ is the local function considered above in the case when $P$ is a boundary point. 
In fact, only finitely many divisors need to be considered. 
Then the above argument can be summarized as follows.

\begin{proposition}\label{prop: extension singularity}
Suppose that $F$ satisfies the condition in Proposition \ref{prop: extension criterion}. 
The holomorphic tensor $\omega(F)$ extends holomorphically over 
a desingularization of ${\FGcpt}$ if $\nu_{E}(F)\geq m$. 
\end{proposition}

\begin{proof}
The condition $\nu_{E}(F)\geq m$ assures that $\omega(F)$ extends over a desingularization of $V_{x}/G$ 
for every $P\in {\rm Sing}({\FGcpt})_{p}$ and cyclic $G<G_{x}$. 
By the cyclic reduction (\cite{Ta} Proposition 3.1), 
$\omega(F)$ extends over a desingularization of $V_{x}/G_{x}$ 
for every $P\in {\rm Sing}({\FGcpt})_{p}$. 
With the RTW criterion (Theorem \ref{thm: RTW}) for singular points not in ${\rm Sing}({\FGcpt})_{p}$, 
we conclude that $\omega(F)$ extends holomorphically over a desingularization of ${\FGcpt}$. 
\end{proof}

As explained after Lemma \ref{lem: vanish order vs RTp}, 
the necessary vanishing order at each exceptional divisor can be improved in some cases 
by taking into account the contribution from $RT_{p}$ and taking explicit toric resolution.


\section{Modular forms and subvarieties}\label{sec: proof main}

In this section we prove the results stated in \S \ref{sec: intro}. 
In \S \ref{ssec: slope} we define various slopes of modular forms. 
In \S \ref{ssec: cusp form criterion} we prove Theorem \ref{thm: intro full}. 
In \S \ref{ssec: proof A} we prove Theorem \ref{thm: any neat}. 
In \S \ref{ssec: pcsg} we study the case of principal congruence subgroups.

\subsection{Slopes}\label{ssec: slope}

Let $L$ be a lattice of signature $(2, n)$ with $n\geq 3$ and ${\G}$ be a finite-index subgroup of ${\OL}$. 
We do not assume ${\G}$ neat nor $-{\rm id}\in {\G}$. 
Let $F$ be a ${\G}$-modular form of weight $k$. 
In this subsection we define the notions of reflective slope, cusp slope, and elliptic slope of $F$.

The reflective slope is the most simple to define. 
Let $R$ be the ramification divisor of ${\D}\to {\FG}$ and 
$R=\sum_{i}R_i$ be its irreducible decomposition. 
Modulo the ${\G}$-action we have only finitely many components. 
Let $\nu_{R}(F)=\min_{i} (\nu_{R_{i}}(F))$ be the minimum of the vanishing order of $F$ along the components of $R$. 
We define the \textit{reflective slope} of $F$ by 
\begin{equation*}
s_{R}(F) = k/\nu_{R}(F)  \quad \in {\Q}_{>0}\cup \{ \infty \}. 
\end{equation*}
We have $s_{R}(F)<\infty$ if and only if $F$ vanishes at $R$. 
We use the word ``reflective'' because each component $R_i$ is defined by a reflection 
in $\langle {\G}, -{\rm id} \rangle$ 
(\cite{GHS1} Corollary 2.13).

To define the cusp slope requires the choice of a toroidal compactification ${\FGcpt}$ of ${\FG}$. 
Let $\Delta \subset {\FGcpt}$ be the boundary divisor and  
$\Delta=\sum_{i}\Delta_{i}$ be its irreducible decomposition.  
Let $\nu_{\Delta}(F)=\min_{i} (\nu_{\sigma_i, geom}(F))$ be 
the minimum of the geometric vanishing order of $F$ along the components $\Delta_{i}$, 
where $\sigma_{i}$ is a ray corresponding to $\Delta_{i}$ and 
$\nu_{\sigma_i, geom}(F)$ is as defined in \eqref{eqn: geom vanish order}. 
Then we define the \textit{cusp slope} of $F$ by 
\begin{equation*}
s_{\Delta}(F) =k/\nu_{\Delta}(F)  \quad \in {\Q}_{>0}\cup \{ \infty \}.  
\end{equation*}
We have $s_{\Delta}(F)<\infty$ if and only if $F$ is a cusp form. 
Note that the value of $\nu_{\Delta}(F)=\nu_{\Delta,\Sigma}(F)$ depends on the choice of $\Sigma$.

To define the elliptic slope, we fix $0 < p \leq n$. 
We consider the interior and the boundary separately.  
In the interior, the definition is similar to the above ones. 
We put 
\begin{equation*}
\nu_{E,int}(F) = \min_{P, G, E_{i}}(\nu_{E_{i}}(F)), 
\end{equation*}
where $P$ ranges over points in ${\rm Sing}({\FG})_{p}$ with quotient model $V_{x}/G_{x}$, 
$G$ ranges over cyclic subgroups of $G_{x}$, 
and $E_{i}$ ranges over exceptional divisors of a toric resolution of $V_{x}/G$. 
Each $\nu_{E_{i}}(F)$ is as defined in \eqref{eqn: vanish order excep div}. 
Then we define the \textit{interior elliptic slope} of $F$ as 
\begin{equation*}
s_{E,int}(F) = k / \nu_{E,int}(F) \quad \in {\Q}_{>0}\cup \{ \infty \}. 
\end{equation*}

At the boundary, the definition is more complicated because 
we need to take into account the construction $F_{x}=F\cdot \prod_{i}q_{i}^{-m}$ in \S \ref{ssec: singularity}, 
where the weight is assumed to be $pm$.  
We take the toroidal compactification ${\FGcpt}$ so that 
each fan $\Sigma_{I}$ is nonsingular with respect to $U(I)_{{\Q}}\cap \langle {\G}, -{\rm id} \rangle$. 
We define the \textit{boundary elliptic slope} of $F$ by 
\begin{equation*}
s_{E, bdr}(F) = 
\max_{P,G,E_{i}} \frac{k}{\nu_{E_{i}}(F)} \left( 1 + \sum_{j=1}^{d} \nu_{E_{i}}(q_{j}) \right) 
\quad \in {\Q}_{> 0} \cup \{ \infty \}, 
\end{equation*}
where $P$ ranges over boundary points in ${\rm Sing}({\FGcpt})_{p}$ with quotient model $V_{x}/G_{x}$, 
$G$ ranges over cyclic subgroups of $G_{x}$, and 
$E_{i}$ ranges over exceptional divisors of a toric resolution of $V_{x}/G$, and 
$q_j=\exp (2\pi i z_j)$ is as in \S \ref{ssec: singularity}. 
If $s_{E,bdr}(F)<p$ (resp.~$\leq p$), 
there exist $\alpha \in {\N}$, $\beta \in 2{\N}$ 
with $\alpha p > \beta k$ (resp.~$\alpha p \geq \beta k$) 
such that 
\begin{equation}\label{eqn: boundary elliptic slope}
\nu_{E_{i}}(F^{\beta}\cdot \prod_{j=1}^{d}q_{j}^{-\alpha}) \geq \alpha 
\end{equation}
for all $P, G, E_{i}$ as above. 
(Write $s_{E,bdr}(F)=k\beta/\alpha$.) 
Here we require $\beta$ to be even since we want $F^{\beta}$ to be $\langle {\G}, -{\rm id} \rangle$-modular. 
Note that $s_{E,bdr}(F)$ depends on the choice of $\Sigma$.

Finally, we define the \textit{elliptic slope} of $F$ by 
\begin{equation*}
s_{E}(F) = \max (s_{E,int}(F), s_{E,bdr}(F)). 
\end{equation*}
Note that this actually depends on the choice of $p$ 
because the points $P$ range over ${\rm Sing}({\FGcpt})_{p}$; 
$s_{E}(F)=s_{E,p}(F)$ is non-increasing as $p$ grows. 
Therefore, if $s_{E,p}(F)\leq p$ and $p<p'$, then $s_{E,p'}(F)< p'$. 
It does not affect the rest of the article even if we suppress this dependence.

\subsection{Low slope modular form criterion}\label{ssec: cusp form criterion}

We can now prove Theorem \ref{thm: intro full}. 
Let $L$ be a lattice of signature $(2, n)$ with $n\geq 3$ and ${\G}$ be a finite-index subgroup of ${\OL}$. 
We fix $0<p\leq n$. 

\begin{theorem}[Theorem \ref{thm: intro full}]\label{thm: full}
We take a toroidal compactification ${\FGcpt}$ of ${\FG}$ such that 
each fan $\Sigma_{I}$ is nonsingular with respect to $U(I)_{{\Q}}\cap \langle {\G}, -{\rm id} \rangle$. 
Suppose that we have a ${\G}$-modular form $F$ with 
$s_{R}(F)<p$, $s_{\Delta}(F)<p$, $s_{E}(F)<p$. 
Let $Y$ be a subvariety of ${\FG}$ with ${\rkY}\geq p$ not contained in the branch locus. 
If $Y\not\subset {\rm div}(F)$, then $Y$ is of Freitag general type. 
If moreover $Y$ is nondegenerate, $Y$ is of general type. 

If we instead require only $s_{R}(F)\leq p$, $s_{\Delta}(F)\leq p$, $s_{E}(F)\leq p$, 
the conclusion is that $Y$ is not rationally connected if $Y\not\subset {\rm div}(F)$. 
\end{theorem}

\begin{proof}
Let $k_0$ be the weight of the given modular form $F$ and let 
\begin{equation*}
a = \min (  \nu_{R}(F),  \nu_{\Delta}(F), \nu_{E, int}(F) ). 
\end{equation*}
The condition that $F$ has slope $<p$ implies $pa > k_0$. 
Moreover, we take $\alpha, \beta$ with $\alpha p > \beta k_{0}$ as in \eqref{eqn: boundary elliptic slope}. 
Let $\alpha_{0}= \min (\alpha, a \beta)$. 
Then $\alpha_{0} p > \beta k_{0}$. 
If we take a sufficiently large even number $m$, 
then $(\alpha_{0}p-\beta k_0)m$ is large enough so that 
we can find ${\G}$-modular forms $F_0, \cdots, F_N$ of weight $(\alpha_{0}p-\beta k_0)m$ with the properties that 
$F_0, \cdots, F_N$ have no common zero point and  
\begin{equation*}
[F_0 : \cdots : F_N]|_{Y} : Y \to {\proj}^N 
\end{equation*}
is a generically finite morphism onto its image. 

The modular form $F_iF^{\beta m}$ has weight $m\alpha_{0}p$ and  
\begin{equation*}
\nu_{\ast}(F_iF^{\beta m})\geq \nu_{\ast}(F^{\beta m}) = \beta m \cdot \nu_{\ast}(F) \geq \beta m a \geq m\alpha_{0} 
\end{equation*}
for $\ast=R, \Delta, (E, int)$. 
Moreover, at the exceptional divisors $E_{\ast}$ over the boundary points in ${\rm Sing}({\FGcpt})_{p}$, 
we have 
\begin{equation*}
\nu_{E_{\ast}}(F_i F^{\beta m} \prod_{j}q_{j}^{-m\alpha_{0}}) \geq 
m \cdot \nu_{E_{\ast}}(F^{\beta} \prod_{j}q_{j}^{-\alpha}) \geq 
m \alpha \geq m \alpha_{0} 
\end{equation*}
by \eqref{eqn: boundary elliptic slope}. 
Then we can apply Propositions \ref{prop: extension criterion} and \ref{prop: extension singularity} 
to the $\langle {\G}, -{\rm id}\rangle$-modular forms $F_iF^{\beta m}$.  
This shows that the corresponding holomorphic tensor 
$\omega_{i}=\omega(F_iF^{\beta m})$ 
of type $(\Omega^{p})^{\otimes \alpha_{0} m}$ 
extends holomorphically over a desingularization of ${\FGcpt}$. 
It follows that the restriction of $\omega_{i}$ to $Y$ extends to 
a holomorphic tensor of type $(\Omega^{p})^{\otimes \alpha_{0}m}$ on a smooth projective model of $Y$. 

Since $Y\not\subset {\rm div}(F)$, 
we have $F_{i}F^{\beta m}|_{Y}\not\equiv 0$ if $F_{i}|_{Y}\not\equiv 0$.  
Since ${\rkY}\geq p$, Lemma \ref{lem: restrict} implies that  
$\omega_{i}|_{Y}\not\equiv 0$ if $F_{i}|_{Y}\not\equiv 0$. 
In particular, $\omega_{0}|_{Y}, \cdots, \omega_{N}|_{Y}$ are not all identically zero.   
The rational map 
\begin{equation*}
[\omega_{0}|_{Y}: \cdots : \omega_{N}|_{Y}] : Y \dashrightarrow {\proj}^N 
\end{equation*}
coincides with the morphism 
\begin{equation*}
[F_0F^{\beta m} : \cdots : F_NF^{\beta m}]|_{Y} = [F_0 : \cdots : F_N]|_{Y} : Y \to {\proj}^N, 
\end{equation*}
and hence is generically finite onto its image. 
This concludes that $Y$ is of Freitag general type. 
When $\dim (Y) = p$, so that $Y$ is nondegenerate, 
then $\omega_{i}|_{Y}$ are pluricanonical forms, and hence $Y$ is of general type. 
The case of nondegenerate $Y$ with $\dim (Y) >p$ is similar. 
(Replace $p$ with $p'=\dim(Y)$.) 

In the case when $F$ has slope $\leq p$, we consider the holomorphic tensor 
$\omega=\omega(F^{2p})$ of type $(\Omega^{p})^{\otimes 2k_{0}}$. 
This extends holomorphically over a desingularization of ${\FGcpt}$ for the same reason as above. 
The restriction $\omega|_Y$ gives a nonzero holomorphic tensor on a smooth projective model of $Y$, 
and thus $Y$ is not rationally connected. 
This completes the proof of Theorem \ref{thm: full}. 
\end{proof}

Next we give a variant of Theorem \ref{thm: full} 
which avoids the elliptic obstruction but instead imposes a condition on $Y$. 
Let $\bar{Y}$ be the closure of $Y$ in ${\FGcpt}$. 
We assume either of the following: 
\begin{enumerate}
\item[(A)] ${\rm Sing}({\FGcpt})_{p} \cap \bar{Y} =\emptyset$. 
\item[(B)] ${\rm Sing}({\FGcpt})_{p} \cap {\rm Sing}(\bar{Y}) = \emptyset$, 
and ${\rm Sing}({\FGcpt}) \cap \bar{Y}$ is of codimension $\geq 2$ in $\bar{Y}$.  
\end{enumerate}
Then we have the following. 

\begin{theorem}\label{thm: criterion full}
Let ${\FGcpt}$ be as in Theorem \ref{thm: full} and  
suppose that we have a ${\G}$-modular form $F$ with $s_{R}(F)<p$ and $s_{\Delta}(F)<p$. 
Let $Y$ be a subvariety of ${\FG}$ of ${\rkY}\geq p$ not contained in the branch locus and satisfying (A) or (B) above.  
If $Y\not\subset {\rm div}(F)$, $Y$ is of Freitag general type. 
If moreover $Y$ is nondegenerate, $Y$ is of general type. 

If we instead require only $s_{R}(F) \leq p$ and $s_{\Delta}(F)\leq p$, 
the conclusion is that 
$Y$ is not rationally connected if $Y\not\subset {\rm div}(F)$. 
\end{theorem}

\begin{proof}
We keep the notation from the proof of Theorem \ref{thm: full}. 
In the present case, $\omega_{i}$ extends over the regular locus of ${\FGcpt}$. 
What has to be shown is that the holomorphic tensor  
$\omega_{i}|_{Y}$ extends holomorphically over a desingularization of $\bar{Y}$. 
In case (A), this is obvious. 
In case (B), $\omega_{i}|_{Y}$ first extends over the regular locus $\bar{Y}_{reg}$ of $\bar{Y}$ because 
$\bar{Y}_{reg}\cap {\rm Sing}({\FGcpt})$ is assumed to be of codimension $\geq 2$ in $\bar{Y}_{reg}$. 
Moreover, ${\rm Sing}(\bar{Y})$ is contained in $U={\FGcpt}-{\rm Sing}({\FGcpt})_{p}$ by assumption. 
We take desingularizations $\tilde{Y}\to \bar{Y}\cap U$ and $\tilde{U}\to U$. 
Then $\omega_{i}$ extends over $\tilde{U}$, and its pullback gives a holomorphic tensor on 
a blow-up of $\tilde{Y}$ that resolves $\tilde{Y}\dashrightarrow \tilde{U}$. 
This shows that $\omega_{i}|_{\bar{Y}\cap U}$ extends holomorphically over $\tilde{Y}$, 
and so $\omega_{i}|_{Y}$ extends holomorphically over a desingularization of $\bar{Y}$. 
\end{proof}

When $n\geq 9$, Gritsenko-Hulek-Sankaran \cite{GHS1} proved that ${\rm Sing}({\FGcpt})_{n}=\emptyset$,  
namely there is no elliptic obstruction for $p=n$. 
In this case, $Y={\FG}$, and Theorem \ref{thm: full} = \ref{thm: criterion full} is due to them. 
However, as $p$ decreases, the locus ${\rm Sing}({\FGcpt})_{p}$ gets larger, 
eventually to ${\rm Sing}({\FGcpt})_{1}={\rm Sing}({\FGcpt})$. 
Even when $p=n-1$, 
some canonical singularities on ${\FG}$ may have $RT_{n-1}<1$, 
as the following example shows.  

\begin{example}\label{ex: singularity}
Suppose that $L$ splits as $L=M\oplus K$ with $K$ negative-definite of rank $2$, 
and that ${\G}$ contains $\gamma={\rm id}_{M}\oplus -{\rm id}_{K}$.  
Then $\gamma$ fixes $\mathcal{D}_{M}={\proj}M_{{\C}}\cap {\D}$. 
If $x\in \mathcal{D}_{M}$, 
$\gamma$ acts on the subspace $T_{x}\mathcal{D}_{M}\subset T_{x}{\D}$ trivially and 
on its normal space by $-1$. 
Hence the $\gamma$-action on $T_{x}{\D}$ has eigenvalue 
$1, \cdots, 1, -1, -1$. 
This shows that 
$RT_{n-1}(\gamma, T_{x}{\D})=1/2$ and 
$RT_{n-2}(\gamma, T_{x}{\D})=0$. 
\end{example}

Thus it may be rather usual for orthogonal modular varieties 
that $RT_{p}<1$ even when $p$ is sufficiently close to $n$. 
This is different from the case of Siegel modular varieties \cite{We}  
where $RT_p\geq 1$ if $p$ is close to $\dim {\D}$. 
For this reason we need to provide a full version that involves the elliptic obstruction (Theorem \ref{thm: full}). 

In practice, to estimate the elliptic obstruction will require a lot of explicit calculations including 
understanding of ${\rm Sing}({\FGcpt})_{p}$ and taking various toric resolutions. 
Moreover, the necessary requirement for the vanishing order can be improved in general, 
as explained at the end of \S \ref{ssec: RTW}.  
See \cite{Da} for such a type of analysis for $p=n=4$, namely pluricanonical forms on non-canonical singularities. 

\begin{remark}
The formulation of Theorem \ref{thm: full} involves the choice of a toroidal compactification ${\FGcpt}$ 
because $s_{E}(F)=s_{E,\Sigma}(F)$ and $s_{\Delta}(F)=s_{\Delta,\Sigma}(F)$ depend on $\Sigma$. 
As pointed out by the referee, when ${\G}$ is neat, we can reformulate the main part of Theorem \ref{thm: full}  
in a way that avoids this dependence as follows. 
In that case, we have no elliptic obstruction nor reflective obstruction. 
We take the infimum $s_{cusp}(F)=\inf_{\Sigma}s_{\Delta,\Sigma}(F)$ of the cusp slopes over all nonsingular $\Sigma$. 
Then the first paragraph of Theorem \ref{thm: full} holds true 
even if we replace the condition $s_{\Delta,\Sigma}(F)<p$ for a specific $\Sigma$ 
by the condition $s_{cusp}(F)<p$ which does not involve the choice of $\Sigma$. 
\end{remark}

\subsection{Proof of Theorem \ref{thm: any neat}}\label{ssec: proof A}

In this subsection we deduce Theorem \ref{thm: any neat} from Theorem \ref{thm: intro full}. 
Let $k(n)=4([(n-2)/8]+3)$. 
In \cite{Ma1} \S 3, it is proved that for any lattice $L$ of signature $(2, n)$ with $n\geq 11$ 
there exists a nonzero cusp form $F$ of weight $k(n)$ with respect to ${\OL}$. 
In particular, $F$ is also a cusp form with respect to any neat subgroup ${\G}$ of ${\OL}$. 
As a ${\G}$-cusp form, $F$ has cusp slope $\leq k(n)$. 
When $n\geq 22$, we have $k(n)<n-1$. 
Then Theorem \ref{thm: any neat} follows by 
applying the neat case of Theorem \ref{thm: intro full} to this cusp form 
and putting $Z=Z_{k(n)+1}\cup Z'_{k(n)}$.  
\qed

\begin{remark}
The fact that sub orthogonal modular varieties of ${\FG}$ of dimension $>k(n)$ are of general type 
can also be proved directly by using restriction of $F$ to each of them. 
The proof of Theorem \ref{thm: full} tells that we can obtain sufficiently many pluricanonical forms 
on various sub orthogonal modular varieties (of fixed dimension) 
as restriction of \textit{common} holomorphic tensors on the ambient variety ${\FG}$.  
\end{remark}

\subsection{Principal congruence subgroups}\label{ssec: pcsg}

In this subsection we give some improvements of Theorem \ref{thm: any neat} 
for an explicit class of neat groups. 
We assume in this subsection that the lattice $L$ is even. 
This is mainly for using the lifting construction in \cite{Gr}, \cite{Bo}. 
Let $L^{\vee}\subset L_{{\Q}}$ be the dual lattice of $L$.  
Let $\Gamma_{L}$ be the subgroup of ${\OL}$ acting trivially on $L^{\vee}/L$. 
$\Gamma_{L}$ is sometimes called the \textit{stable orthogonal group} or the \textit{discriminant kernel}. 
Let $N\geq 1$ be a natural number. 
We define the \textit{principal congruence subgroup} of $\Gamma_{L}$ of level $N$ by 
\begin{equation}\label{eqn: pcsg}
{\GN} = {\rm Ker}( {\OL} \to {\rm GL}(L^{\vee}/NL)). 
\end{equation}
Note that ${\GN}$ can be naturally identified with the stable orthogonal group $\Gamma_{L(N)}$ of the scaled lattice $L(N)$.  
When $N\geq 3$, ${\GN}$ is neat (\cite{Gr} p.1202). 
We write ${\FN}=\mathcal{F}({\GN})$ and take a smooth projective toroidal compactification ${\FN}^{\Sigma}$ of ${\FN}$.

\begin{theorem}\label{thm: pcsg full}
Let $L$ be even and $N\geq 3$. 
Suppose that we have a nonzero $\Gamma_{L}$-cusp form of weight $k$. 
If $Np>k$, 
there exists an algebraic subset $Z\subsetneq {\FN}$ that contains 
all nondegenerate subvarieties $Y$ of non-general type with $\dim (Y) \geq p$ and 
all subvarieties $Y$ with ${\rkY}\geq p$ not of Freitag general type. 
If $Np=k$, 
there exists an algebraic subset $Z'\subsetneq {\FN}$ that contains 
all rationally connected subvarieties $Y$ with ${\rkY}\geq p$.  
\end{theorem}

For the proof of Theorem \ref{thm: pcsg full} we need some preliminaries. 
Let $I\subset L$ be a rank $1$ primitive isotropic sublattice. 
We write $U(I)_{{\Z},N}=U(I)_{{\Q}}\cap {\GN}$ in order to specify the level. 
This is a lattice in the linear space $U(I)_{{\Q}}\simeq L(I)_{{\Q}}$. 
More precisely, 

\begin{lemma}\label{lem: UIZN}
We have $U(I)_{{\Z},N}=N\cdot U(I)_{{\Z},1}=N\cdot L(I)$, 
where $N\cdot$ means the scalar multiplication by $N$ on $U(I)_{{\Q}}\simeq L(I)_{{\Q}}$. 
\end{lemma}

\begin{proof}
The equality $U(I)_{{\Z},1}=L(I)$ for even $L$ is well known (see \cite{Ma2} Lemma 4.1). 
We can check $U(I)_{{\Z},N}=N\cdot L(I)$ in the same way. 
\end{proof}

As a consequence we have 

\begin{lemma}\label{lem: level  vanishing order}
Let $F$ be a $\Gamma_{L}$-cusp form. 
Then, as a ${\GN}$-cusp form, $F$ has vanishing order $\geq N$ along every irreducible component of 
the boundary divisor of ${\FN}^{\Sigma}$. 
\end{lemma}

\begin{proof}
Let $\sigma={\R}_{\geq 0}v$ be a ray in $\Sigma_{I}$, 
where $v$ is a primitive vector of $U(I)_{{\Z},N}$. 
Let $F=\sum_{l}a(l)q^{l}$ be the Fourier expansion of $F$ at $I$. 
Here $l$ runs a priori over $U(I)_{{\Z},N}^{\vee}$, 
but actually it runs over the sublattice $U(I)_{{\Z},1}^{\vee} \subset U(I)_{{\Z},N}^{\vee}$ 
because $F$ is $\Gamma_{L}$-modular. 
Since 
$U(I)_{{\Z},1}^{\vee} = N\cdot U(I)_{{\Z},N}^{\vee}$ 
by Lemma \ref{lem: UIZN}, we see that $(v, l)$ is divisible by $N$. 
Since $F$ is a cusp form, we have $a(l)=0$ when $(v, l)=0$. 
This shows that $(v, l)\geq N$ if $a(l)\ne 0$. 
\end{proof}

Now we can deduce Theorem \ref{thm: pcsg full}. 

\begin{proof}[(Proof of Theorem \ref{thm: pcsg full})]
Let $F$ be a $\Gamma_L$-cusp form of weight $k$. 
By Lemma \ref{lem: level  vanishing order}, $F$ as a ${\GN}$-cusp form has cusp slope $\leq k/N$. 
In the notation of Theorem \ref{thm: intro full}, this means that $Z_{p}\ne {\FN}$ if $p>k/N$ and $Z_{p}'\ne {\FN}$ if $p=k/N$. 
Then Theorem \ref{thm: pcsg full} follows from the neat case of Theorem \ref{thm: intro full}. 
\end{proof}

We give some concrete applications of Theorem \ref{thm: pcsg full}. 
Recall that the \textit{Witt index} of $L$ is the maximal rank of isotropic sublattices of $L$. 
When $n\geq 5$, the Witt index is always $2$.  

\begin{corollary}\label{thm: pcsg}
Suppose that $L$ has Witt index $2$. 
When $Np>k(n)$, there exists an algebraic subset $Z\subsetneq {\FN}$ that contains 
all nondegenerate subvarieties $Y$ of non-general type with $\dim Y\geq p$ and 
all subvarieties $Y$ with ${\rkY}\geq p$ not of Freitag general type. 
When $Np=k(n)$, 
there exists an algebraic subset $Z'\subsetneq {\FN}$ that contains 
all rationally connected subvarieties $Y$ with ${\rkY}\geq p$.  
\end{corollary}

\begin{proof}
We take a maximal even overlattice $L'$ of $L$. 
Let $U$ be the even unimodular lattice of signature $(1, 1)$. 
Since every primitive isotropic vector $l$ in $L'$ satisfies $(l, L')={\Z}$ by the maximality of $L'$ (see \cite{Ni}), 
it comes from an embedding $U\hookrightarrow L'$. 
Hence $L'$ contains $2U$.  
Then we can use the results of \cite{Ma1} \S 3 to see that 
there exists a nonzero $\Gamma_{L'}$-cusp form of weight $k(n)$. 
Since $\Gamma_{L}\subset \Gamma_{L'}$ by \cite{Ni}, 
this gives a $\Gamma_{L}$-cusp form of the same weight. 
Corollary \ref{thm: pcsg} follows by applying Theorem \ref{thm: pcsg full} to this cusp form. 
\end{proof}

Corollary \ref{thm: pcsg} contains the following special cases: 
\begin{itemize}
\item If $N>k(n)$, all nondegenerate subvarieties of non-general type are contained in $Z\subsetneq {\FN}$, 
i.e., the Lang conjecture holds for nondegenerate subvarieties of ${\FN}$. 
\item If $N>k(n)$, all subvarieties $Y$ with ${\rkY}>0$ (e.g., $\dim (Y) >n/2$) that are not of Freitag general type 
are contained in $Z\subsetneq {\FN}$. 
\item Let $n\geq 7$ and $N\geq 3$. 
Then ${\FN}$ contains only finitely many divisors that are not of Freitag general type. 
\end{itemize}

The bound on $Np$ can be improved if we have a $\Gamma_{L}$-cusp form of smaller weight. 

\begin{example}\label{ex: K3}
Let $L_{2d}= 2U \oplus 2E_8 \oplus \langle -2d \rangle$. 
Then $\mathcal{F}_{2d}=\mathcal{F}(\Gamma_{L_{2d}})$ is the moduli space of polarized $K3$ surfaces of degree $2d$. 
When $N$ is coprime to $2d$, 
the cover $\mathcal{F}_{2d}(N)=\mathcal{F}(\Gamma_{L_{2d}}[N])$ parametrizes $K3$ surfaces of degree $2d$ with level $N$ structure. 
By Gritsenko-Hulek-Sankaran (\cite{GHS2} Proposition 3.1), 
there exists a $\Gamma_{L_{2d}}$-cusp form of weight $10$ when $d\geq 181$, 
and of weight $13$ when $d\geq 5$, $d\ne 6$. 
(See \cite{GHS2} for smaller $d$.) 
Therefore the Lang conjecture holds for nondegenerate subvarieties of $\mathcal{F}_{2d}(N)$ 
when $N\geq 11$ for $d\geq 181$, 
and when $N\geq 14$ for $d\geq 5$, $d\ne 6$. 
This gives an improvement of the general Corollary \ref{thm: pcsg} for $L_{2d}$ where $k(19)=20$. 
\end{example}

In Corollary \ref{thm: pcsg} and Example \ref{ex: K3}, 
the $\Gamma_{L}$-cusp forms are constructed as the Gritsenko-Borcherds additive lifting (\cite{Gr}, \cite{Bo}) 
of cusp forms of type $\rho_L$, where $\rho_{L}$ is the Weil representation of $L$. 
The following proposition generalizes Example \ref{ex: K3} and improves Corollary \ref{thm: pcsg} when $|\det L \, |$ is large. 

\begin{proposition}\label{prop: improve by dim formula}
Let $L$ be an even lattice of signature $(2, n)$ containing $2U$ with $|\det L \, |$ sufficiently large. 
When $N > 2[n/4]+2$, the Lang conjecture holds for nondegenerate subvarieties of ${\FN}$. 
\end{proposition}

\begin{proof}
By a result of Bruinier-Ehlen-Freitag \cite{BEF}, if $|\det L \, |$ is sufficiently large, 
there exists a nonzero cusp form of type $\rho_L$ and weight $l$ with 
$3/2 \leq l \leq 3$ and $l\equiv 1-n/2$ mod $2{\Z}$. 
Since $L$ contains $2U$, the lifting is injective (\cite{Gr}) and produces a nonzero $\Gamma_{L}$-cusp form 
of weight $l+n/2-1=2[n/4]+2$. 
Then our assertion follows from Theorem \ref{thm: pcsg full}. 
\end{proof}


\end{document}